\documentclass[12pt,reqno]{amsart}
\usepackage{amssymb,amsmath,amsthm,fullpage,enumerate}
\usepackage{xcolor}
\usepackage{biblatex} %Imports biblatex package
\usepackage{comment}
\usepackage{cases}

\usepackage[OT2,OT1]{fontenc}

\DeclareFontFamily{OT2}{cmr}{\hyphenchar\font45 }
\DeclareFontShape{OT2}{cmr}{m}{n}{<->wncyr10}{}
\DeclareFontShape{OT2}{cmr}{m}{it}{<->wncyi10}{}
\DeclareFontShape{OT2}{cmr}{m}{sc}{<->wncysc10}{}
\DeclareFontShape{OT2}{cmr}{b}{n}{<->wncyb10}{}
\DeclareFontShape{OT2}{cmr}{bx}{n}{<->ssub*wncyr/b/n}{}

\DeclareFontFamily{OT2}{cmss}{\hyphenchar\font45 }
\DeclareFontShape{OT2}{cmss}{m}{n}{<->wncyss10}{}

\DeclareRobustCommand\cyr{\fontencoding{OT2}\selectfont}
\DeclareTextFontCommand{\textcyr}{\cyr}

\DeclareUnicodeCharacter{0306}{}

\newtheorem{theorem}{Theorem}[section]
\newtheorem{prop}[theorem]{Proposition}

\newtheorem{corr}[theorem]{Corollary}

\theoremstyle{definition}

\newtheorem{deff}[theorem]{Definition}

\theoremstyle{remark}
\newtheorem{comm}[theorem]{Remark}

\newcommand{\R}{\mathbb R}

\newcommand{\p}{\partial}

\newcommand{\z}{\bar z}

\newcommand{\dbar}{\bar\partial}

\DeclareMathOperator{\im}{Im}
\DeclareMathOperator{\re}{Re}

\newcommand{\bc}{\mathbb{B}}

\title{Bicomplex Schwarz and Dirichlet Boundary Value Problems}

\author{William L. Blair}
\address{Department of Mathematics\\
  The University of Texas at Tyler\\
  Tyler, TX 75799}
\email{wblair@uttyler.edu}



\begin{document}

\begin{abstract}
    We define and solve boundary value problems of Schwarz and Dirichlet type on the complex unit disk for bicomplex-valued functions.
\end{abstract}

\maketitle

\section{Introduction}

In this paper, we construct bicomplex versions of the Schwarz and Dirichlet boundary value problems on the complex unit disk and show they are uniquely solvable by formulas similar to those used for the associated complex boundary value problems. 

The Dirichlet boundary value problem is a classic boundary value problem of mathematical analysis. The problem seeks a harmonic function on a domain that agrees with a prescribed function on the boundary. On the complex unit disk, the Dirichlet problem is known to be uniquely solvable for real-valued functions on the boundary so long as the function is Lebesgue integrable on the unit circle by integrating the boundary value against the Poisson kernel for the disk. In \cite{Straube}, E. Straube showed that if the boundary value function is replaced with a distribution, then the corresponding Dirichlet problem is uniquely solvable by now considering the function that is realized by pairing the boundary distribution against the Poisson kernel. 

The Schwarz boundary value problem is a complex generalization of the Dirichlet problem where now we seek a holomorphic function in the interior of the disk that has real part that agrees with a prescribed boundary function. This is immediately solvable by using the solution to the Dirichlet problem where the boundary condition is with respect to the boundary value for the real part of the Schwarz problem solution. This is a harmonic function, and since the disk is simply connected, it follows that the harmonic conjugate exists. Hence, one can construct a holomorphic function with real part that agrees with the desired boundary value. However, this is not well defined, as harmonic conjugates are only unique up to arbitrary constants. In this case, the problem is made well defined by observing that the harmonic conjugate of the constructed Poisson integral is zero at the origin. Consequently, if one prescribes the value of the imaginary part of the solution to the Schwarz problem at the origin, then a harmonic conjugate is constructable with that specific constant and this problem is uniquely solvable. See \cite{Beg}. As in the case of the Dirichlet problem, this construction is immediately generalizable to distributional boundary values by pairing the boundary distribution against the holomorphic completion of the Poisson kernel. This solution for the generalized Schwarz boundary value problem was shown to be unique in \cite{WBD}. Also in \cite{WBD}, the author and B. Delgado showed that a Schwarz boundary value problem that seeks a solution to a nonhomogeneous Cauchy-Riemann equation and has boundary condition with respect to a distribution is uniquely solvable by combining the described construction for solving the holomorphic Schwarz boundary value problem and the known solution formula for the nonhomogeneous Schwarz problem with continuous boundary condition found in \cite{Beg}.

The bicomplex numbers are a higher-dimensional generalization of the complex numbers that is different from $\mathbb{C}^2$, as studied in the analysis of functions of several complex variables, and the quaternions, as studied in the analysis of functions of a quaternionic variable. Functions of a single complex variable that take values in the bicomplex numbers are known to be useful in studying the complex stationary Schr\"odinger equation, see \cite{KravAPFT}, and were studied previously in, for example, \cite{BCAtomic, BCHoiv, BCHarmVek, BCBergman, BCTransmutation}. Since solutions of the homogeneous bicomplex Cauchy-Riemann type equation, where the differential operator is the natural bicomplexification of the complex Cauchy-Riemann operator studied in \cite{BCAtomic, BCHoiv, BCHarmVek, BCBergman, BCTransmutation}, are representable as a linear combination of a complex holomorphic and a complex anti-holomorphic function, it follows that there is a natural path to considering boundary value problems traditionally associated with complex Cauchy-Riemann type equations in the setting of bicomplex numbers. Specifically, we define Schwarz and Dirichlet boundary value problems for bicomplex-valued function of a single complex variable. We take advantage of the representation of these functions as a linear combination of members of more familiar classes of functions and use the known solution formulas for the corresponding complex boundary value problems to build bicomplex-valued functions that solve these problems. This is the first time these types of boundary value problems have been considered in the bicomplex setting.

We outline the paper. In Section \ref{BackgroundSection}, we provide definitions and background results used in the sections that follow. In Section \ref{bcschwarz}, we define a bicomplex Schwarz boundary value problem and show that it is uniquely solvable by a formula that is the natural bicomplex analogue of the solution to the complex Schwarz boundary value problem. In Section \ref{bcdirichlet}, we define a bicomplex Dirichlet boundary value problem and show the problem is uniquely solvable by a natural extension of the solution formula of the complex Dirichlet boundary value problem.

\section{Background}\label{BackgroundSection}

We use $D$ for the set of complex numbers with modulus less than one and $\p D$ for its boundary. For $p$ a positive real number, we denote by $L^p(D,\mathbb{C})$ the set of complex-valued functions defined on $D$ with integrable modulus raised to the $p^{\text{th}}$ power. We use $C(\p D, \mathbb{R})$ to indicate the set of continuous real-valued functions defined on $\p D$, $C^\infty(\p D)$ to indicate the set of infinitely differentiable functions on $\p D$,  $\mathcal{D}'(\p D)$ is the collection of distributions on $\p D$, and $Hol(D)$ is the set of complex-valued holomorphic functions on $D$.

Now, we provide some definitions of objects that will be used throughout. 

\begin{deff}
    Let $f: D \to \mathbb{C}$. We say that $f$ has a boundary value in the sense of distributions $f_b$, also called a distributional boundary value, if, for every $\gamma \in C^\infty(\p D)$, the limit
    \[
        \lim_{r \nearrow 1} \int_0^{2\pi} f(re^{i\theta}) \gamma(\theta) \,d\theta 
    \]
    exists. 
\end{deff}

The classic Schwarz boundary value problem seeks a holomorphic function in the disk with a prescribed real part on the circle. This is the simplest form of the classic Riemann-Hilbert problem and is a holomorphic extension of the Dirichlet problem for harmonic functions. This problem is not well-defined without the inclusion of requiring the imaginary part of the function to take a prescribed value at the origin. Many variations of this problem have been studied that include the function solving generalizations of the Cauchy-Riemann equations, the boundary conditions being in different classes, and the consideration of other domains. See \cite{BegBook, Beg, higherupper, WB, WB2, WB3, WBD, SchwarzRing, Gaertner, polyhardy, metahardy, bvpuhp, CampMez} and many others. For example, the next theorem shows that the Schwarz problem for the nonhomogensous Cauchy-Riemann equation with continuous boundary condition is solvable. 

\begin{theorem}[Theorem 2.1 \cite{Beg}]\label{Begfirstorder}
The Schwarz boundary value problem 
\[
    \begin{cases}
        \frac{\p w}{\p\z} = f, & \text{ in } D\\
        \re\{w\} = \gamma, & \text{ on } \p D\\
        \im\{w(0)\} = c,
    \end{cases}
\]
for $f \in L^1(D, \mathbb{C})$, $\gamma \in C(\p D, \mathbb{R})$, and $c \in \mathbb{R}$, is uniquely solved by 
\[
w(z) = ic + \frac{1}{2\pi i} \int_{|\zeta| = 1} \gamma(\zeta) \frac{\zeta + z}{\zeta - z} \frac{d\zeta}{\zeta} - \frac{1}{2 \pi} \iint_D \left(\frac{f(\zeta)}{\zeta} \frac{\zeta+z}{\zeta - z} + \frac{\overline{f(\zeta)}}{\overline{\zeta}} \frac{1+z\overline{\zeta}}{1-z\overline{\zeta}} \right)\, d\xi\,d\eta.
\]
\end{theorem}

\begin{comm} Note that
\[
\frac{1}{2\pi i} \int_{|\zeta| = 1} \gamma(\zeta) \frac{\zeta + z}{\zeta - z} \frac{d\zeta}{\zeta} = \frac{1}{2\pi}\int_0^{2\pi} \gamma(e^{it}) \left( P_r(\theta - t) + iQ_r(\theta - t)\right)\,dt,
\]
where $z = re^{i\theta}$, and
\[
P_r(\theta) = \frac{1-r^2}{1-2r\cos(\theta) + r^2},
\]
and 
\[
Q_r(\theta) =  \frac{2r\sin(\theta)}{1-2r\cos(\theta) + r^2}
\]
are the Poisson kernel and the conjugate Poisson kernel on $D$, respectively. 
\end{comm}

In \cite{WBD}, the author with B. Delgado used the characterization of pairing a distribution on the circle with the Poisson kernel of the disk for harmonic functions with distributional boundary values, as found in \cite{Straube}, to show the Schwarz boundary value problem, with the boundary condition of a distributional boundary value, is solved similarly by pairing the distributional boundary value against the holomorphic extension of the Poisson kernel. This result is the next theorem.

\begin{theorem}[Theorem 6.9 \cite{WBD}]\label{nonhomogsbvp}
The Schwarz boundary value problem
\[\begin{cases}
\frac{\p w}{\p\z} = f, &\text{ in } D,\\
\re\{w_b\} = g,\\
\im\{w(0)\} = c,
\end{cases}\]
for $f \in L^1(D)$, $g\in \mathcal{D}'(\p D)$, and $c\in \mathbb{R}$, is solved by 
\[
w = \frac{1}{2\pi}\langle g, P_r(\theta - \cdot) + iQ_r(\theta - \cdot)\rangle + ic  - \frac{1}{2\pi}\iint_{|\zeta|<1} \left( \frac{f(\zeta)}{\zeta}\,\frac{\zeta + z}{\zeta - z} + \frac{\overline{f(\zeta)}}{\overline{\zeta}}\,\frac{1+z\overline{\zeta}}{1-z\overline{\zeta}}\right)\,d\xi\,d\eta,
\]
and the solution is unique. 
\end{theorem}

By iteration of the formula from Theorem \ref{nonhomogsbvp}, the following theorem generalizes the last to higher-order nonhomogeneous Cauchy-Riemann equations.

\begin{theorem}[Theorem 6.11 \cite{WBD}]\label{higherschwarz}
The Schwarz boundary value problem 
\[\begin{cases}
\frac{\p^n w}{\p\z^n} = f  & \text{ in } D\\
\re\left\{\left(\frac{\p^k w}{\p\z^k}\right)_b\right\} = h_k, & 0 \leq k \leq n -1 \\
\im\left\{\frac{\p^k w}{\p\z^k}(0)\right\} = c_k, & 0 \leq k \leq n -1, 
\end{cases}\]
for $n \in \mathbb{N}$, $f \in L^1(D, \mathbb{C})$, $h_k \in \re\{(H^{p_k}(D))_b\}$ where $p_k > \frac{1}{2}$ for $k = 1, 2, \ldots, n-1$, $h_0 \in \mathcal{D}'(\p D)$, and $c_k \in \mathbb{R}$ for $k = 0, 1, 2,\ldots, n-1$, is uniquely solved by the function 
\begin{align*}
w(z) &= i\sum_{k = 0}^{n-1}\frac{c_k}{k!}(z+\z)^k  + \sum_{k=0}^{n-1}\frac{(-1)^k}{2\pi   k!} \langle h_k, (P_r(\theta - \cdot)+iQ_r(\theta - \cdot))(e^{i(\cdot)} - re^{i\theta}+\overline{e^{i(\cdot)} - re^{i\theta}})^k\rangle \\
&\quad - \frac{1}{2\pi}\iint_{|\zeta|<1} \left( \frac{f(\zeta)}{\zeta}\,\frac{\zeta + z}{\zeta - z} + \frac{\overline{f(\zeta)}}{\overline{\zeta}}\,\frac{1+z\overline{\zeta}}{1-z\overline{\zeta}}\right)(\zeta - z+\overline{\zeta - z})^{n-1}\,d\xi\,d\eta.
\end{align*}

\end{theorem}

\begin{comm}
See Remark \ref{firstorderhardycomment} for the definition of $\re\{(H^{p}(D))_b\} \subset \mathcal{D}'(\p D)$. Also, Theorem \ref{higherschwarz} reads exactly the same if the $h_k$, $0\leq k \leq n-1$ are continuous functions on $\p D$, and in that case, the pairings of the distributions against the kernel are proper integrals. For $k \geq 1$, the distributions $h_k$ must be real parts of distributional boundary values of functions in complex-holomorphic Hardy spaces to guarantee that the function constructed by pairing the distribution against the Poisson kernel plus $i$ times the conjugate Poisson kernel is an $L^1(D, \mathbb{C})$ function, see Theorem 6.2 of \cite{GHJH2} and Theorem 5.11 of \cite{Duren}. In the absence of this guarantee, the iteration scheme used to construct the solution function may not make sense. 
\end{comm}

The bicomplex numbers $\bc$ are the set $\mathbb{C}^2$ with the usual addition and multiplication by scalars. The multiplication operation between two elements $z = (z_1, z_2), w=(w_1,w_2) \in \bc$ is defined by 
\[
        zw = (z_1,z_2) (w_1,w_2) = (z_1w_1-z_2w_2, z_1w_2 + z_2w_1).
\]
Using the imaginary unit $j$, $z = (z_1, z_2) \in \bc$ is representable as $z = z_1 + jz_2$. Using this representation, the multiplication rules allows multiplication to computationally resemble the multiplication between complex numbers, i.e., 
\[
    zw = (z_1 + jz_2) (w_1 + jw_2) = z_1w_1-z_2w_2 + j( z_1w_2 + z_2w_1).
\]
In particular (and in contrast to the quaternions), multiplication of bicomplex number is a commutative operation. See \cite{BCAtomic, BCHoiv, BCHarmVek, PriceMulti, CastaKrav, BCTransmutation, BCBergman, FundBicomplex, KravAPFT, BCHolo,ComplexSchr} for more background on the bicomplex numbers and their role in analysis.

Next, we provide definitions and known results concerning the bicomplex numbers that we use in the sequel. 

\begin{prop}[Proposition 1 \cite{FundBicomplex}]\label{bcidemrep}
    For every $z \in \bc$, there exist unique $z^+, z^- \in \mathbb{C}$ such that
    \[
        z = p^+ z^+  + p^- z^-,
    \]
    where $p^\pm := \frac{1}{2} \left( 1 \pm ji \right)$.
\end{prop}

Note that the bicomplex numbers $p^\pm$ are nonzero zero divisors and idempotent elements of the bicomplex numbers. With the above representation, we define a norm on the bicomplex numbers. 

\begin{deff}\label{bcnormdeff}
    For every $z \in \bc$ with idempotent representation $z = p^+ z^+  + p^- z^-$, $z^\pm \in \mathbb{C}$, we define the norm $||\cdot||_{\bc}$ by 
    \[  
        ||z||_{\bc} := \sqrt{\frac{|z^+|^2 + |z^-|^2}{2}}.
    \]  
\end{deff}

Using this norm, we define the bicomplex version of the Lebesgue spaces.

\begin{deff}
    For $0 < p < \infty$, we define $L^p(D,\bc)$ to be the collection of functions $f: D \to\bc$ such that
    \[
        ||f||_{L^p_\bc} := \left( \iint_D ||f(z)||_{\bc}^p\,dx\,dy\right)^{1/p} < \infty.
    \]
\end{deff}

\begin{deff}\label{bcdbardeff}
We define the bicomplex differential operators $\partial$ and $\dbar$ as
\[
    \partial :=  p^+\frac{\p}{\p \z} + p^- \frac{\p}{\p z}.
\]
and
\[
    \dbar :=  p^+\frac{\p}{\p z} + p^- \frac{\p}{\p \z},
\]
where $\frac{\p}{\p z}$ and $\frac{\p}{\p\z}$ are the complex Wirtinger derivatives. 
\end{deff}

\begin{deff}
     We say a function $f: D \to\bc$ is $\bc$-holomorphic whenever
     \[
        \dbar f = 0.
     \]
\end{deff}  

Differential equations involving these bicomplex differential operators were considered in, for example, \cite{BCAtomic, BCHoiv, BCHarmVek, BCTransmutation, BCBergman, FundBicomplex}. With this definition, the idempotent representation leads to a useful relationship between $\bc$-holomorphic functions and the complex holomorphic functions.

\begin{prop}[Remark 3 \cite{BCAtomic} (or \cite{FundBicomplex, BCTransmutation, BCBergman})]\label{holoiffbcholo}
    For a function $w = p^+w^+ + p^- w^- : D \to \bc$, $\dbar w = 0$ if and only if $\overline{w^+}, w^- \in Hol(D)$. 
\end{prop}

\begin{comm}
    In contrast to some of the literature concerning bicomplex numbers, we continue to use $\overline{(\cdot)}$ to indicate the usual complex conjugation throughout. We do this to aid readers more familiar with the study of functions of complex variables and since we have no need for bicomplex conjugation (in any of its forms) in the sequel. 
\end{comm}

From \cite{BCTransmutation, BCBergman, FundBicomplex}, the integral operator
\[
\mathcal{T}(f) := p^+ \left(-\frac{1}{2 \pi} \iint_D \frac{f^+(\zeta)}{\overline{\zeta - z}} \,d\xi\,d\eta \right) + p^- \left( -\frac{1}{2 \pi} \iint_D \frac{f^-(\zeta)}{\zeta - z}\,d\xi\,d\eta\right),
\]
for $f \in L^1(D,\bc)$, satisfies
\[
    \dbar \mathcal{T}(f) = f
\]
because
\[
    \frac{\p}{\p z}\left(-\frac{1}{2 \pi} \iint_D \frac{f^+(\zeta)}{\overline{\zeta - z}} \,d\xi\,d\eta \right) = f^+(z)
\]
and 
\[
    \frac{\p}{\p \z}\left( -\frac{1}{2 \pi} \iint_D \frac{f^-(\zeta)}{\zeta - z}\,d\xi\,d\eta\right) = f^-(z)
\]
where $\zeta = \xi + i \eta$. See \cite{BegBook, Beg, hoio} for more information about these integral operators in the complex setting.

We now define the integral operator $T_\bc$, acting on functions $f \in L^1(D,\bc)$, by 
\[
    T_\bc(f)(z):= p^+ T_*(f^+)(z) + p^- T(f^-)(z),
\]
where 
\[
T(f^-)(z) := -\frac{1}{2 \pi} \iint_D \left(\frac{f^-(\zeta)}{\zeta} \frac{\zeta+z}{\zeta - z} + \frac{\overline{f^-(\zeta)}}{\overline{\zeta}} \frac{1+z\overline{\zeta}}{1-z\overline{\zeta}} \right)\,d\xi\,d\eta
\]
and
\[
T_*(f^+)(z) := -\frac{1}{2 \pi} \iint_D \left(\frac{f^+(\zeta)}{\overline{\zeta}} \frac{\overline{\zeta+z}}{\overline{\zeta - z}} + \frac{\overline{f^+(\zeta)}}{\zeta} \frac{1+\z\zeta}{1-\z\zeta} \right) \,d\xi\,d\eta.
\]
Note that for $\sigma : D \to \mathbb{C}$, $T_*(\sigma) = T(\sigma^*)^*$. Thus, $\frac{\p }{\p \z} T(f^-) = f^-$, and $\frac{\p }{\p z} T(f^+) = f^+$, by Theorem \ref{Begfirstorder}. Therefore,
\[
    \dbar T_\bc(f) = f,
\]
and by construction (see Theorem \ref{Begfirstorder}), $T_\bc(f)$ satisfies the following conditions:
\[
    \begin{cases}
        \re\{T_\bc(f)^+\}|_{\p D} = \re\{T_*(f^+)\}|_{\p D}= 0\\
        \re\{T_\bc(f)^-\}|_{\p D} = \re\{T(f^-)\}|_{\p D} = 0\\
        \im\{T_\bc(f)^+(0)\} = \im\{T_*(f^+)(0)\}= 0\\
        \im\{T_\bc(f)^-(0)\} = \im\{T(f^-)0\} = 0\\
    \end{cases}.
\]
Also, we can iterate and have
\[
    T^n_\bc(f)(z) = p^+T^n_*(f^+)(z) + p^- T^n(f^-)(z),
\]
for every positive integer $n$, and 
\[
 \dbar^m T^n_\bc(f) = 
    \begin{cases}
       f, & n = m\\
       T^{n-m}_\bc(f), & m < n\\
       \dbar^{m - n} f, & m > n
    \end{cases}.
\] 

Therefore, $T_\bc$ is an integral operator that behaves in an analogous way to the Cauchy-type integral from Theorem \ref{nonhomogsbvp} in the bicomplex setting. That is, the operator acts as a right-inverse to $\dbar$, satisfies similar boundary and pointwise conditions at the origin, and maintains this behavior after iteration.

\section{Bicomplex Schwarz Boundary Value Problem}\label{bcschwarz}

We consider a bicomplex analogue of the classic Schwarz boundary value problem from complex analysis. See \cite{BegBook, Beg} for background about this problem in the complex setting.  We show that the problem in the bicomplex setting is solvable, solutions are unique, and give a constructive representation of the solution. 

\begin{theorem}\label{bcholoSchwarzthm}
For $b_1, b_2 \in C(\p D, \R)$ and $c_1, c_2 \in \R$, the bicomplex-Schwarz boundary value problem 
\[
\begin{cases}
    \dbar w = 0\\
    \re\{w^+\}|_{\p D} = b_1 \\
    \re\{w^-\}|_{\p D} = b_2\\
    \im\{ w^+(0)\}  = c_1 \\
    \im\{ w^-(0) \} = c_2
\end{cases}
\]
is uniquely solved by 
\begin{align*}
w(z) &= p^+ \left(\frac{1}{2\pi i }\int_{|\zeta| = 1} b_1(\zeta) \frac{\overline{\zeta + z}}{\overline{\zeta - z}} \frac{d\zeta}{\zeta} + ic_1 \right) + p^- \left(\frac{1}{2\pi i }\int_{|\zeta| = 1} b_2(\zeta) \frac{\zeta + z}{\zeta - z} \frac{d\zeta}{\zeta} +ic_2 \right) .
\end{align*}

\end{theorem}

\begin{proof}
Observe that 
\[
    f(re^{i\theta}) := \frac{1}{2\pi i }\int_{|\zeta|=1} b_2(\zeta) \frac{\zeta + z}{\zeta - z} \frac{d\zeta}{\zeta} +ic_2
\]
uniquely solves the complex-Schwarz boundary value problem
\[
\begin{cases}
    \frac{\p f}{\p \z} = 0\\
    \re\{f\}|_{\p D} = b_2\\
    \im\{ f(0) \} = c_2
\end{cases},
\]
by Corollary \ref{Begfirstorder}. Also, the function 
\[
    g(re^{i\theta}) :=  \frac{1}{2\pi i }\int_{|\zeta| = 1} b_1(\zeta) \frac{\zeta + z}{\zeta - z} \frac{d\zeta}{\zeta} - ic_1
\]
uniquely solves the complex-Schwarz boundary value problem
\[
\begin{cases}
    \frac{\p g}{\p \z} = 0\\
    \re\{g\}|_{\p D} = b_1\\
    \im\{ g(0) \} = -c_1
\end{cases},
\]
by Corollary \ref{Begfirstorder}. Hence, $\overline{g}$ solves 
\[
\begin{cases}
    \frac{\p \overline{g}}{\p z} = 0\\
    \re\{\overline{g}\}|_{\p D} = b_1\\
    \im\{ \overline{g(0)} \} = c_1
\end{cases},
\]
see Corollary 2.2 of \cite{Beg}. Define $w:= p^+ \overline{g} + p^- f$.  By Proposition \ref{holoiffbcholo}, $\dbar w = 0$ if and only if $\overline{w^+}, w^- \in Hol(D)$. Since $\overline{w^+}, w^- \in Hol(D)$, it follows that $\dbar w= 0$. Note also that $w^+$ and $w^-$ satisfy the required boundary and pointwise conditions by their construction. Since the idempotent representation is unique, see Proposition \ref{bcidemrep}, and $w$ solves the boundary value problem, it follows that $\overline{g}$ and $f$ are the unique $w^+$ and $w^-$, respectively, of a solution $w$. Since $w^\pm$ uniquely determine $w$, it follows that $w$, as constructed, is the unique solution to the boundary value problem. 
\end{proof}

\begin{theorem}\label{bcnonhomogfirstorderthm}
For $b_1, b_2 \in C(\p D, \R)$, $c_1, c_2 \in \R$, and $f \in L^1(D, \bc)$, the nonhomogeneous bicomplex-Schwarz boundary value problem 
\[
\begin{cases}
    \dbar w = f\\
    \re\{w^+\}|_{\p D} = b_1 \\
    \re\{w^-\}|_{\p D} = b_2\\
    \im\{ w^+(0)\}  = c_1 \\
    \im\{ w^-(0) \} = c_2
\end{cases}
\]
is uniquely solved by 
\begin{align*}
w(z) &= p^+ \left(\frac{1}{2\pi i }\int_{|\zeta|=1} b_1(\zeta) \frac{\overline{\zeta + z}}{\overline{\zeta - z}} \frac{d\zeta}{\zeta} + ic_1  \right)  \\
&\quad\quad + p^- \left(\frac{1}{2\pi i }\int_{|\zeta|=1} b_2(\zeta) \frac{\zeta + z}{\zeta - z} \frac{d\zeta}{\zeta} +ic_2 \right) +T_\bc(f)(z).
\end{align*}

\end{theorem}

\begin{proof}
This is a direct application of Theorem \ref{bcholoSchwarzthm} and the properties of the operator $T_\bc$ discussed in Section \ref{BackgroundSection}.
\end{proof}

By appealing to Theorem \ref{nonhomogsbvp}, we generalize Theorem \ref{bcnonhomogfirstorderthm} so that the boundary condition is with respect to a distributional boundary value. This is the roughest boundary condition where the Schwarz boundary value problem is known to be solvable.

\begin{theorem}\label{nonhomogbcdistrbvfirstorder}
For $b_1, b_2 \in \mathcal{D}'(\p D)$, $c_1, c_2 \in \R$, and $f \in L^1(D, \bc)$, the bicomplex-Schwarz boundary value problem 
\[
\begin{cases}
    \dbar w = f\\
    \re\{w^+_b\} = b_1 \\
    \re\{w^-_b\} = b_2\\
    \im\{ w^+(0)\}  = c_1 \\
    \im\{ w^-(0) \} = c_2
\end{cases}
\]
is uniquely solved by 
\begin{align*}
w(z) &= p^+ \left( \overline{\left[\frac{1}{2\pi} \langle b_1, P_r(\theta-\cdot) + i Q_r(\theta - \cdot)\rangle\right]} + ic_1  \right)  \\
&\quad\quad + p^- \left(\frac{1}{2\pi} \langle b_2, P_r(\theta-\cdot) + i Q_r(\theta - \cdot)\rangle +ic_2 \right) +T_\bc(f)(z).
\end{align*}

\end{theorem}

\begin{proof}
By Theorem \ref{nonhomogsbvp}, 
\[
    g(re^{i\theta}) := \frac{1}{2\pi} \langle b_1, P_r(\theta - \cdot) + iQ_r(\theta - \cdot)\rangle -ic_1
\]
uniquely solves
\[
    \begin{cases}
        \frac{\p g}{\p \z} = 0\\
        \re\{g_b\} = b_1\\
        \im\{g(0)\} = -c_1
    \end{cases},
\]
and 
\[
    f(re^{i\theta}) := \frac{1}{2\pi} \langle b_2, P_r(\theta - \cdot) + iQ_r(\theta - \cdot)\rangle + ic_2
\]
uniquely solves
\[
    \begin{cases}
        \frac{\p f}{\p \z} = 0\\
        \re\{f_b\} = b_2\\
        \im\{f(0)\} = c_2
    \end{cases}.
\]
Hence, 
\[
\overline{g(re^{i\theta})} = \overline{\left[\frac{1}{2\pi} \langle b_1, P_r(\theta - \cdot) + iQ_r(\theta - \cdot)\rangle \right]} + ic_1
\]
uniquely solves
\[
\begin{cases}
        \frac{\p \overline{g}}{\p z} = 0\\
        \re\{\overline{g_b}\} = b_1\\
        \im\{\overline{g(0)}\} = c_1
    \end{cases},
\]
and
\[
    \tilde{w} := p^+ \overline{g} + p^- f
\]
uniquely solves
\[
\begin{cases}
    \dbar \tilde{w} = 0\\
    \re\{\tilde{w}^+_b\} = b_1 \\
    \re\{\tilde{w}^-_b\} = b_2\\
    \im\{ \tilde{w}^+(0)\}  = c_1 \\
    \im\{ \tilde{w}^-(0) \} = c_2
\end{cases}.
\]
Therefore, 
\[
    w := \tilde{w} + T_\bc(f)
\]
uniquely solves 
\[
\begin{cases}
    \dbar w = f\\
    \re\{w^+_b\} = b_1 \\
    \re\{w^-_b\} = b_2\\
    \im\{ w^+(0)\}  = c_1 \\
    \im\{ w^-(0) \} = c_2
\end{cases}.
\]

\end{proof}

\begin{comm}\label{firstorderhardycomment}
    If $f \equiv 0$ and $b_1, b_2$ in the statement of the theorem above are the real parts of distributional boundary values of functions in a complex holomorphic Hardy space
    \[
        H^p(D) := \{ h \in Hol(D) : \sup_{0 < r < 1 } \int_0^{2\pi} |h(re^{i\theta})|^p\,d\theta < \infty\},
    \]
    i.e., $b_1, b_2 \in \re\{(H^{p}(D))_b\} := \{ \re\{h\}_b : h \in H^p(D)\}$, $ 0 < p < \infty$, (see \cite{Duren, Koosis, Rep} for background on these spaces), then $\overline{w^+}, w^- \in H^p(D)$, by Theorem 6.2 in \cite{GHJH2}. By Theorem 4.1 in \cite{BCAtomic}, this implies that $w$ is an element of the $\bc$-holomorphic Hardy space 
    \[
        H^p(D,\bc) := \{h: D \to \bc :  \dbar h = 0 \text{ and }\sup_{0 < r < 1 } \int_0^{2\pi} ||h(re^{i\theta})||_\bc^p\,d\theta < \infty\}
    \]
    studied in \cite{BCAtomic}. If $f \not\equiv 0$, then by the same argument, for $0 < p < \infty$, $f \in L^q(D,\bc)$, $q>2$ or $1< q \leq 2$ and $p < \frac{q}{2-q}$, and $b_1, b_2 \in \re\{(H^p(D))_b\}$, $\overline{w^+}$ is in the generalized complex Hardy class
    \[
    H^p_{\overline{f^+}}(D) := \{g: D \to \mathbb{C} : \frac{\p g}{\p\z} = f^+ \text{ and } \sup_{0 < r < 1 } \int_0^{2\pi} |g(re^{i\theta})|^p\,d\theta < \infty\},
    \]
     see \cite{WB, BCAtomic} for more background about these classes of functions, and $w^-$ is an element of the similarly defined $H^p_{f^-}(D)$.  This implies that $w$ is in the generalized Hardy class of $\bc$-valued functions 
     \[
        H^p_f(D,\bc) := \{h: D \to \bc :  \dbar h = f \text{ and }\sup_{0 < r < 1 } \int_0^{2\pi} ||h(re^{i\theta})||_\bc^p\,d\theta < \infty\},
    \]
    which are also studied in \cite{BCAtomic}. See also Remark 2.14 in \cite{WB2}.
\end{comm}

Now, we extend the first-order result from above to the natural higher-order generalization.

\begin{theorem}
For $n$ a positive integer, $b^\pm_0 \in \mathcal{D}'(\p D)$, $b^\pm_k\in \re\{(H^{p^\pm_k}(D))_b\}$, for $1\leq k \leq n-1$ where $p^\pm_k > \frac{1}{2}$, $c^\pm_k \in \R$, for $0 \leq k \leq n-1$, and $f \in L^1(D, \bc)$, the bicomplex-Schwarz boundary value problem 
\[
\begin{cases}
    \dbar^n w = f\\
    \re\{(\dbar ^k w)_b^+\} = b^+_k \\
    \re\{(\dbar^k w)_b^-\} = b^-_k\\
    \im\{ (\dbar^k w)^+(0)\}  = c^+_k \\
    \im\{ (\dbar^k w)^-(0) \} = c^-_k
\end{cases}
\]
is uniquely solved by 
\[
w(z) = p^+ w^+ + p^- w^-,
\]
where 
\begin{align*}
    \overline{w^+(z)} &:=  -i\sum_{k = 0}^{n-1}\frac{c^+_k}{k!}(z+\z)^k   + \sum_{k=0}^{n-1}\frac{(-1)^k}{2\pi   k!} \langle b^+_k, (P_r(\theta - \cdot)+iQ_r(\theta - \cdot))(e^{i(\cdot)} - re^{i\theta}+\overline{e^{i(\cdot)} - re^{i\theta}})^k\rangle  \\
&\quad\quad - \frac{1}{2\pi}\iint_{|\zeta|<1} \left( \frac{\overline{f^+(\zeta)}}{\zeta}\,\frac{\zeta + z}{\zeta - z} + \frac{f^+(\zeta)}{\overline{\zeta}}\,\frac{1+z\overline{\zeta}}{1-z\overline{\zeta}}\right)(\zeta - z+\overline{\zeta - z})^{n-1}\,d\xi\,d\eta 
\end{align*}
and
\begin{align*}
    w^-(z) &:= i\sum_{k = 0}^{n-1}\frac{c^-_k}{k!}(z+\z)^k  + \sum_{k=0}^{n-1}\frac{(-1)^k}{2\pi   k!} \langle b^-_k, (P_r(\theta - \cdot)+iQ_r(\theta - \cdot))(e^{i(\cdot)} - re^{i\theta}+\overline{e^{i(\cdot)} - re^{i\theta}})^k\rangle \\
&\quad - \frac{1}{2\pi}\iint_{|\zeta|<1} \left( \frac{f^-(\zeta)}{\zeta}\,\frac{\zeta + z}{\zeta - z} + \frac{\overline{f^-(\zeta)}}{\overline{\zeta}}\,\frac{1+z\overline{\zeta}}{1-z\overline{\zeta}}\right)(\zeta - z+\overline{\zeta - z})^{n-1}\,d\xi\,d\eta
\end{align*}

\end{theorem}

\begin{proof}
The formula and its uniqueness is a direct result of iterating Theorem \ref{nonhomogbcdistrbvfirstorder} or by appeal to Theorem \ref{higherschwarz} to construct the appropriate $w^+$ and $w^-$.

\end{proof}

\begin{comm}
    Similarly to the first order case (see Remark \ref{firstorderhardycomment}), if $f \in L^q(D,\bc)$, $q>2$ and $b_0^\pm \in \re\{(H^{p^\pm_0}(D))_b\}$, $0 < p^\pm_0 < \infty$, then $\overline{w^+}$ is an element of the generalized Hardy class 
    \[
        H^{n,p^+}_{\overline{f^+}}(D) := \{ g: D \to \mathbb{C}: \frac{\p^n g}{\p \z^n} = \overline{f^+} \text{ and } \sum_{k = 0}^{n-1} \sup_{0 < r < 1} \int_0^{2\pi}\left| \frac{\p^k g}{\p \z^k}(re^{i\theta})\right|^{p^+}\,d\theta < \infty\},
    \]
    where $p^+ := \min_{0 \leq k \leq n-1}\{p^+_k\}$, and $w^-$ is an element of the similarly defined $H^{n,p^-}_{f^-}(D)$, where $p^-:= \min_{0 \leq k \leq n-1}\{p^-_k\}$. See \cite{WB, BCAtomic} for more background on these generalized Hardy classes. This implies that $w$ is an element of the generalized bicomplex Hardy class 
    \[
        H^{n,p}_f(D, \bc) := \{ h: D \to\bc : \dbar^n w = f \text{ and } \sum_{k = 0}^{n-1} \sup_{0 < r < 1} \int_0^{2\pi} ||\dbar^k h(re^{i\theta})||_{\bc}^p \,d\theta < \infty\},
    \]
    where $p := \min\{p^+, p^-\}$. These classes of functions were previously studied in \cite{BCAtomic}. Also see Remark 2.14 in \cite{WB2}. 

\end{comm}

g\section{Bicomplex Dirichlet Boundary Value Problem}\label{bcdirichlet}

In this final section, we consider a bicomplex analogue of the Dirichlet boundary value problem. Observe that
\begin{align*}
    4\partial \dbar &:= 4\left(   p^+\frac{\p}{\p \z} + p^- \frac{\p}{\p z}\right)  \left( p^+\frac{\p}{\p z} + p^- \frac{\p}{\p \z}\right)\\
    &= 4\left(p^+\frac{\p}{\p\z}\frac{\p}{\p z} + p^- \frac{\p}{\p z}\frac{\p}{\p\z} \right)\\
    &= p^+ \Delta + p^- \Delta \\
    &= \Delta.
\end{align*}
So, the Dirichlet problems we consider are, as in the classical case, with respect to harmonic functions. The novelty as presented is that the functions we seek are $\bc$-valued. 

\begin{theorem}
    The bicomplex Dirichlet problem
    \[
        \begin{cases}
            \p \dbar f = 0\\
            f|_{\p D} = 0
        \end{cases}
    \]
has only the trivial solution.
\end{theorem}

\begin{proof}
By Proposition \ref{bcidemrep}, $f: D \to \bc$ is harmonic if and only if $f^+$ and $f^-$ are harmonic, and $f|_{\p D} = 0$ if and only if $f^+|_{\p D} = 0 = f^-|_{\p D}$. So, $f^\pm$ are harmonic functions in the disk that are zero on the circle. Therefore, $f^\pm \equiv 0$, and consequently, $f \equiv 0$.
\end{proof}

\begin{theorem}
    For $g \in L^1(\p D, \bc)$, the bicomplex Dirichlet problem
    \[
        \begin{cases}
            \p \dbar f = 0\\
            f|_{\p D} = g
        \end{cases}
    \]
is uniquely solved by 
\[
    f = p^+ \frac{1}{2\pi} \int_0^{2\pi} g^+(e^{it})P_r(\theta - t) \,dt  + p^- \frac{1}{2\pi} \int_0^{2\pi} g^-(e^{it})P_r(\theta - t) \,dt.
\]
\end{theorem}

\begin{proof}
By Proposition \ref{bcidemrep}, $f: D \to \bc$ is harmonic if and only if $f^+$ and $f^-$ are harmonic, and if $f|_{\p D} =  p^+ f^+|_{\p D} + p^- f^-|_{\p D} = g = p^+ g^+ + p^- g^-$, then $f^+|_{\p D} = g^+ \in L^1(\p D)$ and $f^-|_{\p D} = g^- \in L^1(\p D)$. Thus, 
\[
    f^+(re^{i\theta}) = \frac{1}{2\pi} \int_0^{2\pi} g^+(e^{i\theta}) P_{r}(\theta -t)\,dt,
\]
\[
    f^-(re^{i\theta}) = \frac{1}{2\pi} \int_0^{2\pi} g^-(e^{i\theta}) P_{r}(\theta -t)\,dt,
\]
and these solutions are unique. Therefore, 
\[
f(re^{i\theta}) = p^+ \frac{1}{2\pi} \int_0^{2\pi} g^+(e^{it})P_r(\theta - t) \,dt  + p^- \frac{1}{2\pi} \int_0^{2\pi} g^-(e^{it})P_r(\theta - t) \,dt.
\]
\end{proof}

\begin{corr}
    The bicomplex Dirichlet problem
    \[
        \begin{cases}
            \p \dbar f = 0\\
            f_b = g
        \end{cases}
    \]
for $g \in \mathcal{D'}(\p D, \bc) := \{ h \in \mathcal{D}'(\p D): \langle h, \varphi \rangle \in \bc, \text{ for }\varphi \in C^\infty(\p D)\}$, is uniquely solved by 
\[
    f = p^+ \frac{1}{2\pi} \langle g^+, P_r(\theta - \cdot) \rangle + p^- \frac{1}{2\pi} \langle g^-, P_r(\theta - \cdot)  \rangle.
\]
\end{corr}

\begin{proof}
The proof of this result is the exact same as the last theorem with the Poisson integral replaced with the distributional pairing against the Poisson kernel. The distributional pairing against the prescribed distributional boundary value is shown to be the unique harmonic function with that distributional boundary value in \cite{Straube}. Therefore, the formula in the statement is the unique solution to the bicomplex Dirichlet problem by an appeal to Proposition \ref{bcidemrep}. 
\end{proof}

%\section*{Acknowledgments}
%\color{red} *** Fill-in *** \color{black}

\printbibliography

\section*{Funding}

The authors declare that no funds, grants, or other support were received during the preparation of this manuscript.

\section*{Competing Interests}

The authors have no relevant financial or non-financial interests to disclose.

\section*{Author Contributions}

All authors contributed equally in all aspects of the preparation of this article. All authors read and approved the final manuscript.

\section*{Data Availability}

No data was produced during the creation of this article.

\end{document}